\documentclass{siamltex}
\usepackage{amsmath,amssymb,amsfonts,latexsym,graphicx,color,multirow,footnote}

\newtheorem{algorithm}{Algorithm}

\title{A contour-integral based method for counting the eigenvalues inside a region in the complex plane}

\author{Guojian Yin\thanks{Shenzhen Institutes of Advanced Technology, Chinese Academy of Science, Shenzhen, P.R. China ({\tt guojianyin@gmail.com}).} }
\begin{document}
\maketitle

\begin{abstract}
In many applications, the information about the number of eigenvalues inside a given region is required. In this paper, we propose a contour-integral based method
for this purpose. The new method is motivated by two findings. There exist methods for estimating the number of eigenvalues inside a region in the complex plane. But our method is able to compute the number of eigenvalues inside the given region exactly. An appealing feature of our method is that it can integrate with the recently developed contour-integral based eigensolvers  to help them detect whether all desired eigenvalues are found. Numerical experiments are reported to show the viability of our new method.
\end{abstract}



\section{Introduction}
Consider the generalized eigenvalue problem
\begin{equation}\label{eq:1-1}
 A {\bf x} = \lambda  B {\bf x},
\end{equation}
where  $ A,  B \in {\mathbb C}^{n \times n}$. The scalars $\lambda \in \mathbb{C}$
and the associated nonzero vectors ${\bf x} \in {\mathbb C}^n$  are
called the eigenvalues and eigenvectors, respectively. Let $\mathcal{D}$ be a disk in the complex plane and enclosed by circle $\Gamma$. In this paper,
we want to develop an approach for computing the number of eigenvalues of (\ref{eq:1-1}) inside $\mathcal{D}$ exactly. Due to the M\"obius transformation, the resulting approach can be adapted to the union of intersections of arbitrary half plane and (complemented) disks, and so a rather general region \cite{BDG97}.

In many applications, it is required to know the number of eigenvalues inside a prescribed region in the complex plane \cite{KP11, LYS14}. For example, it is a prerequisite of the eigensolvers based on divide-and-conquer techniques \cite{BDDRV00, DPS13}. To get this number, the most straightforward way is to compute all eigenvalues inside the target region by some method, such as the rational Krylov subspace method \cite{Ruhe98}. However, this way is always time-consuming and not effective, because we are only interested in the number of eigenvalues rather than the eigenvalues themselves.  

When $A$ and $B$ are Hermitian matrices with $B$ being positive definite, i.e., $zB-A$ is a definite matrix pencil \cite{BDDRV00}, it is well-known that the eigenvalues of (\ref{eq:1-1}) are real-valued and lie on the real line \cite{saad, stewart}. Assume that we want to know the number of eigenvalues inside interval $[a, b]$. The standard method works as follows. Compute two LDL decompositions: $A-aB = L_aD_aL_a^*$ and $A-bB = L_bD_bL_b^*$. Then the difference between the numbers of negative entries in the diagonal of $D_a$ and $D_b$ is exactly the number of eigenvalues inside $[a, b]$. The derivation of this method takes advantage of the Sylvester law of inertia, see \cite{DPS13, GLS94,  saad} for more details. Obviously, the efficiency of this method depends on the accurate computation of two related LDL decompositions \cite{DPS13}. Computing the LDL decomposition requires floating point operations of order $\mathcal{O}(n^3)$ \cite{gvl};  as a result, it becomes computationally prohibitive for large-scale problems. When (\ref{eq:1-1}) comes to the non-Hermitian case, to the best of our knowledge, there is no specific method for exactly computing the number of eigenvalues of (\ref{eq:1-1}) inside $\mathcal{D}$.

The contour-integral based eigenslovers \cite{polizzi,  ss, ST07, YCY14} are recent efforts for computing the eigenvalues inside a prescribed region.
The information about the number of eigenvalues inside the region of interest is crucial to the practical implementation of these eigensolvers. Due to this, recently some methods based on contour-integral were proposed to get estimations of this number. Some of this kind of methods make use of the stochastic estimation of the trace of spectral operator associated with the eigenvalues inside the given region \cite{DPS13, futa, YCY14}. However, they may be unreliable in some cases, for instance when the matrices $A$ and $B$ are ill-conditioned. In \cite{TP13}, another kind of contour-integral based estimate method was presented under the assumption that matrices $A$ and $B$ are Hermitian and $B$ is positive definite. It should be pointed out that all these existing contour-integral based methods always just provide approximations for the exact number of eigenvalues inside the given region.

In this paper, we present a contour-integral based method, which can exactly compute the number of eigenvalues of (\ref{eq:1-1}) inside $\mathcal{D}$. The derivation of the proposed method requires the eigenvalues of (\ref{eq:1-1}) are semi-simple, namely, there are $n$  independent eigenvectors, which is always the case in practical situations. Our new method is motivated by two findings. The first finding comes from using the Gauss-Legendre quadrature rule to approximately compute the spectral operator constructed by a particular contour integral. More details will be discussed in Section 2. The second one is devoted to avoiding the computation of the Weierstrass canonical form of matrix pencil $zB-A$ when using the first finding to count the eigenvalues inside $\mathcal{D}$. We will detail the second finding in Section 4. Since our new method is also based on contour integral, it keeps the promising features of the usual contour-integral based eigensolvers, such as having a good potential to be implemented on a high-performance parallel architecture.  Moreover, it can integrate with the contour-integral based eigensolvers \cite{polizzi,  ss, ST07, YCY14} to help them determine whether all desired eigenvalues are found when they stop.

The paper is organized as follows. In Section 2, we present the first finding, which is derived from using the Gauss-Legendre quadrature rule to approximately compute the  spectral projector constructed by an integral contour. Since our method needs the help of a technique proposed in \cite{YCY14},  we briefly describe the technique in Section 3. In Section 4, we first detail the second finding, and then give the resulting method for counting the eigenvalues inside $\mathcal{D}$. Numerical experiments are reported in Section 5 to show the viability of our new method.

Throughout the paper, the following notation and terminology are used. The real part of a complex number $a$ is denoted by $\Re(a)$. We use $\sqrt{-1}$ to denote the imaginary unit. The subspace spanned by the columns of matrix $X$ is denoted by ${\rm span}\{ X\}$. The rank of $X$ is denoted by $\rank(X)$. The algorithms are presented in \textsc{Matlab} style.

\section{Approximate spectral projector}\label{sec:feast}
Our discussion starts with the spectral projector associated with the eigenvalues inside $\Gamma$, which is constructed by a contour integral defined as
\begin{equation}\label{eq:2-2-1}
Q = \frac{1}{2\pi\sqrt{-1}} \oint_\Gamma  (zB-A)^{-1}B  dz.
\end{equation}

A matrix pencil $ z  B-A$ is called regular if
${\rm det}(zB-A)$ is not identically zero for all $z \in \mathbb{C}$ \cite{demmel, MS73}. Below is a generalization of the Jordan canonical form to the regular matrix pencil case.

\begin{theorem}[Weierstrass canonical form \cite{G59}]
Let $zB-A$ be a regular matrix pencil of order $n$. Then there exist nonsingular matrices
$S$ and $T \in \mathbb{C}^{n\times n}$ such that
\begin{equation}\label{eq:8-3-1}
TAS = \begin{bmatrix}
  J_d    & 0   \\
   0   & I_{n-d}
\end{bmatrix}  \quad {\rm and} \quad TBS= \begin{bmatrix}
   I_d   & 0   \\
  0    & N_{n-d}
\end{bmatrix},
\end{equation}
where $J_d$ is a $d\times d$ matrix in Jordan canonical form
with its diagonal entries corresponding to the eigenvalues of $zB-A$, $N_{n-d}$ is an $(n-d)\times (n-d)$ nilpotent matrix also in Jordan canonical form, and $I_d$ denotes the identity matrix of order $d$.
\end{theorem}

Assume that there are $n$ independent eigenvectors, which implies $J_d$ is a diagonal matrix and $N_{n-d}$ is a zero matrix. Let $J_d = \diag\{\lambda_1,\lambda_2,\ldots, \lambda_d\}$, with $\lambda_i$ being the (finite) eigenvalues \cite{demmel}. Here the $\lambda_i$
are not necessarily distinct and can be repeated according to
their multiplicities.

For $z \ne \lambda_i$, 
the matrix $(zI_d-J_d)$ is invertible. Hence, according to (\ref{eq:8-3-1}), the resolvent operator
\begin{eqnarray}\label{eq:8-3-4}
(zB-A)^{-1}B & = & S\begin{bmatrix}
  (zI_d-J_d)^{-1}    &   0 \\
    0  &  (zN_{n-d}-I_{n-d})^{-1}
\end{bmatrix} TB \nonumber \\
 & = & S\begin{bmatrix}
  (zI_d-J_d)^{-1}    &   0 \\
    0  &  (zN_{n-d}-I_{n-d})^{-1}
\end{bmatrix} \begin{bmatrix}
  I_d    &   0 \\
    0  &  N_{n-d}
\end{bmatrix}S^{-1}\nonumber \\
 & = & SD(z)S^{-1},
\end{eqnarray}
where
\begin{equation}\label{eq:2-1-25}
D(z) = \begin{bmatrix}
  (zI_d-J_d)^{-1}    &   0 \\
    0  & 0
\end{bmatrix},
\end{equation}
and the  diagonal block $(zI_d-J_d)^{-1}$ is of the form:
\begin{equation}\label{equ:7-26-3}
(zI_d-J_d)^{-1}= \left[ \begin{array}{cccc}
 \displaystyle{\frac{1}{z-\lambda_1}}& 0 & \cdots &
 %\displaystyle{\frac{1}{(z-\lambda_i)^{n_i-1}}}&
0\\
0 & \displaystyle{\frac{1}{z-\lambda_2}} & \cdots & %\displaystyle{\frac{1}{(z-\lambda_i)^{n_i-2}}}&
0\\
\vdots& \vdots& \ddots & %\vdots&
\vdots\\
%0& 0& \cdots &\displaystyle{\frac{1}{z-\lambda_i}} & \displaystyle{\frac{1}{(z-\lambda_i)^2}}\\
0 & 0 & \cdots %&0
&\displaystyle{\frac{1}{z-\lambda_d}}\\
\end{array}
\right].
\end{equation}

Assume that there are $s$ eigenvalues enclosed by $\Gamma$, without loss of generality,  let them be $\{\lambda_1, \ldots, \lambda_s\}$.
Then, according to the residue theorem in complex analysis \cite{Rudin}, it follows from (\ref{eq:8-3-4})--(\ref{equ:7-26-3})  that
\begin{equation}\label{equ:7-26-4}
 Q  = S\left[\frac{1}{2\pi\sqrt{-1}} \oint_\Gamma D(z)dz \right]S^{-1}=  S\left[ \begin{array}{cc}
  I_s & 0\\
0 & 0
 \end{array} \right]  S^{-1} = S_{(:, 1:s)}(S^{-1})_{(1:s, :)}.
\end{equation}
It is easy to verify that $ Q^2 =  Q$, which implies that $Q$ is a spectral projector onto  the eigenspace ${\rm span}\{S_{(:, 1:s)}\}$ corresponding to $\{\lambda_1,\lambda_2,\ldots, \lambda_s\}$ \cite{YCY14}.

In view of (\ref{equ:7-26-4}), the spectral projector $Q$ can be obtained via computing the Weierstrass canonical form (cf. (\ref{eq:8-3-1})). However, it is well-known that the Weierstrass canonical form is not suitable for numerical computation \cite{BDDRV00, demmel}. According to the expression (\ref{eq:2-2-1}), an alternative way is to compute $Q$ numerically by a quadrature scheme. In our method, the quadrature scheme is restricted to the Gauss-Legendre quadrature rule \cite{DR84}. Let $c$ and $\rho$ be the center and the radius of circle $\Gamma$, respectively. Applying the $q$-point Gauss-Legendre quadrature rule to (\ref{eq:2-2-1}) yields
\begin{equation}\label{eq:8-4}
 Q  \approx \widetilde{Q}= \frac{1}{2} \sum^{q}_{j=1}\omega_j(z_j-c)(z_j  B- A)^{-1} B = S\left[\frac{1}{2} \sum^{q}_{j=1}\omega_j(z_j-c)D(z_j)\right]S^{-1}.
\end{equation}
Here $z_j=c+\rho e^{\sqrt{-1}\theta_j}$, $\theta_j=(1+ t_j)\pi$, and
$t_j$ is the $j$th Gaussian node with associated weight $\omega_j$. We remark that $\widetilde{Q}$ is an approximate spectral projector. Let
\begin{equation}\label{eq:8-1-1}
D = \frac{1}{2} \sum^{q}_{j=1}\omega_j(z_j-c)D(z_j).
\end{equation}
Comparing to (\ref{equ:7-26-4}) and (\ref{eq:8-4}), we see that $D$  is an approximation to $\left[ \begin{array}{cc}
  I_s & 0\\
0 & 0
 \end{array} \right]$.

Let $\mu = c+r e^{\sqrt{-1} \theta}$, where $r \in [0, \infty)$ and $\theta \in (-\pi, \pi]$. Define
\begin{equation}\label{eq:2-1-1}
\psi (\mu) = \frac{1}{2\pi \sqrt{-1}} \oint_\Gamma\frac{1}{z-\mu}dz.
\end{equation}
According to the residue theorem, we know that $\psi (\mu) = 1$ when $\mu$ is located inside $\Gamma$, and $\psi (\mu) = 0$ when $\mu$ is located outside $\Gamma$. If $\psi( \mu)$ is computed approximately by the $q$-point Gauss-Legendre quadrature rule, then we have
\begin{eqnarray}
\psi(\mu) & \approx &\widetilde{\psi}(\mu)= \frac{1}{2}\sum^q_{j=1} \omega_j (z_j- c)\frac{1}{z_j-\mu}\nonumber \\
 & = & \frac{1}{2}\sum^q_{j=1} \omega_j \frac{-\rho \cos(t_j\pi)-\sqrt{-1} \rho \sin(t_j \pi)}{-(\rho \cos (t_j\pi) +r \cos \theta)- \sqrt{-1} (\rho \sin (t_j \pi) +r \sin \theta)} \label{eq:2-1-2}.
\end{eqnarray}
It was shown that $\widetilde{\psi}(\mu) \geqslant \frac{1}{2}$ if $\mu$ is real-valued and located inside $\Gamma$ \cite{TP13}. This observation sheds light on the following theorem.

%Below we give the first important finding of our work.

\begin{theorem}\label{Th:2-1-3}
If $\mu$ is enclosed by $\Gamma$, then the real part of $\widetilde{\psi}(\mu)$ satisfies
\begin{equation*}\label{eq:2-1-9}
\Re[\widetilde{\psi}(\mu)]> \dfrac{1}{2}.
\end{equation*}
If $\mu$ is located outside $\Gamma$, we have
$$\Re[\widetilde{\psi}(\mu)]< \dfrac{1}{2}.$$
\end{theorem}
\begin{proof}
According to (\ref{eq:2-1-2}), we have

\begin{equation}\label{eq:2-1-13}
\Re[\widetilde{\psi}( \mu)] = \frac{1}{2}\sum^q_{j=1} \omega_j\frac{\rho^2+\rho r \cos (t_j \pi -\theta)}{\rho^2 +r^2+2 \rho r \cos(t_j \pi -\theta)}.
\end{equation}
Let
\begin{equation*}\label{eq:2-1-14}
g_j(r, \theta) = \frac{\rho^2+\rho r \cos (t_j \pi -\theta)}{ \rho^2 +r^2+2 \rho r \cos(t_j \pi -\theta)}.
\end{equation*}
For any given $j$, one can show that
\begin{equation}\label{eq:5-8-1}
g_j(r, \theta)- \frac{1}{2}= \frac{\rho^2-r^2}{2\left[ (\rho +r\cos(t_j \pi -\theta))^2+(r \sin(t_j \pi -\theta))^2 \right]}
\end{equation}
Note that the denominator of the right hand in (\ref{eq:5-8-1}) is always positive. It is readily to see that 
$$g_j(r, \theta)>  \frac{1}{2}$$
for  $r \in [0, \rho)$ and $\theta \in (-\pi, \pi]$, in which case $\mu$ is enclosed by $\Gamma$, and 
$$g_j(r, \theta)<  \frac{1}{2}$$
for  $r \in (\rho, +\infty]$ and $\theta \in (-\pi, \pi]$, in which case $\mu$ is located outside $\Gamma$.

On the other hand, it is well-known that $\sum^q_{j=1} \omega_j=2$ \cite{DR84, TP13}. Therefore,
\begin{equation*}\label{eq:2-1-19}
\Re[\widetilde{\psi}(\mu)] = \frac{1}{2}\sum^q_{j=1} \omega_j g_j(r, \theta) > \frac{1}{2}
\end{equation*}
when $\mu$ is enclosed by $\Gamma$, and 
$$\Re[\widetilde{\psi}(\mu)] < \frac{1}{2}$$
 when $\mu$ is located outside $\Gamma$. 
\end{proof}

We use Figure \ref{Fig:2-1} to depict the function $\Re[\widetilde{\psi}(\mu)]$.  The figure well demonstrates the conclusion in Theorem \ref{Th:2-1-3}.

\begin{figure}
\begin{center}
\includegraphics[width=13cm]{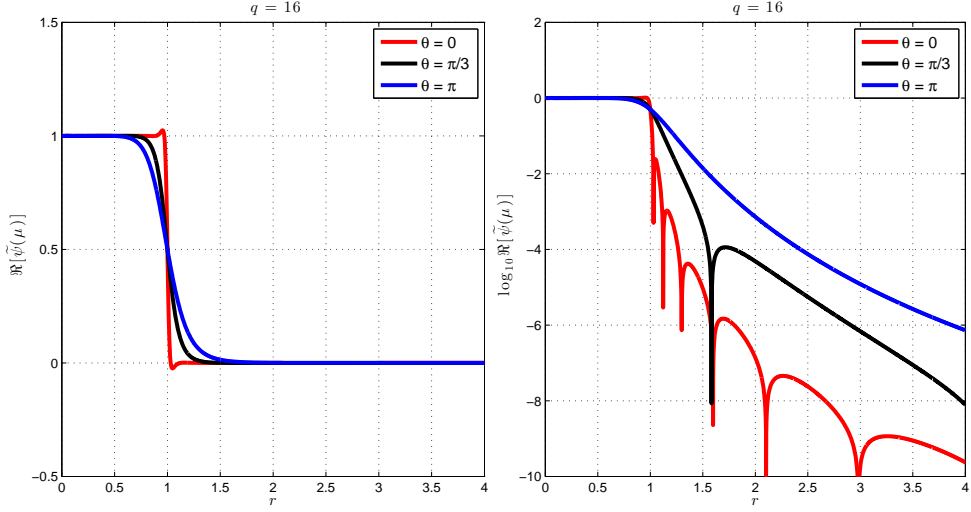}
\caption{This figure illustrates the function $\Re[\widetilde{\psi}(\mu)]$ (cf. (\ref{eq:2-1-13})) when $r \in [0, 4]$. The circle $\Gamma$ has radius $\rho = 1$ with center at the origin. We set the degree of Gauss-Legendre quadrature rule $q =16$. The left picture shows the general shape of the function, and the right one shows the logarithmic scale.}
\label{Fig:2-1}
\end{center}
\end{figure}

Observe (\ref{eq:2-1-25}), (\ref{equ:7-26-3}), and (\ref{eq:8-1-1})--(\ref{eq:2-1-2}), we see that
\begin{equation*}\label{eq:15-2-11}
D_{(i, i)}=\widetilde{\psi}(\lambda_i),\quad i = 1, \ldots d.
\end{equation*}
Due to this, in the following we always refer to $D_{(i, i)}$ as the diagonal entry of $D$ that corresponds to eigenvalue $\lambda_i, i = 1,\ldots, d$.
In viewing of Theorem \ref{Th:2-1-3}, the diagonal entries of $D$ can be divided into two categories, that is,
 $\{\Re[D_{(i,i)}]> \frac{1}{2}\}_{i = 1}^s$ on one group and $\{\Re[D_{(j,j)}]< \frac{1}{2}\}_{ j = s+1}^n$ on the other. This fact is the first finding of our work, which implies that the number $s$ can be obtained by counting the diagonal entries of $D$ whose real parts are larger than $\frac{1}{2}$. According to (\ref{eq:8-1-1}), getting $D$ requires to compute the Weierstrass canonical form of $zB-A$. However, as was suggested in \cite{BDDRV00}, the Weierstrass canonical form is not suitable for numerical computation.

In our work, we present an alternative method to obtain $\{D_{(i, i)}\}_{i=1}^s$, which does not need to compute the Weierstrass canonical form of $zB-A$. As a result, by exploiting the first finding, we can exactly compute the number of eigenvalues of (\ref{eq:1-1}) inside $\Gamma$. The alternative method is our second finding, since it needs the help of a technique proposed in \cite{YCY14}, we briefly introduce the technique in next section before presenting the second finding.

\section{Finding an upper bound for the number of eigenvalues inside $\Gamma$}\label{sec:number}
In \cite{YCY14},  a method based on contour integral  was proposed for finding a good upper bound for the number of eigenvalues inside a given region. Meanwhile, it can produce an approximate projection onto the eigenspace corresponding to the eigenvalues inside the given region.

The method first uses a stochastic estimation of the trace of spectral projector $Q$ (cf.  (\ref{eq:2-2-1})) to obtain an estimation of $s$. By $Y_p\sim {\sf N}_{n\times p}(0,1)$, we mean that $Y_p$ is an $n\times p$ random matrix with independent and identically
distributed entries drawing from standard normal distribution ${\sf N}(0,1)$.  By (\ref{equ:7-26-4}), one can easily verify that
\begin{eqnarray}
\frac1p{\mathbb E}[{\rm trace} (Y_p^*  Q Y_p)]&=& {\rm trace}( Q) =
{\rm trace}( S_{(:,1:s)} ( S^{-1})_{(1:s,:)})\nonumber \\
 &=&{\rm trace}(( S^{-1})_{(1:s,:)}  S_{(:,1:s)}) = {\rm trace}( I_s)=s. \nonumber \label{trace}
\end{eqnarray}
Therefore, 
\begin{equation}\label{eq:5-17-1}
s_0:=\lceil \frac1p{\rm trace} (Y_p^*QY_p)\rceil
\end{equation}
 provides an estimation for $s$, see \cite{futa, SFT, YCY14} for more details.
With this knowledge on hand, a method was then given in \cite{YCY14} to seek a
good upper bound $s_1$ of $s$. Wanting to derive the method, we need the following lemma.

\begin{lemma}[\cite{YCY14}]\label{lem:8-8-1}
Let $Y\in\mathbb{R}^{n\times m}$. If the entries of
$Y$ are continuous random numbers that are independent and identically
distributed (i.i.d.), then the matrix
$( S^{-1})_{(1:m,:)}Y$ is almost surely nonsingular.
\end{lemma}

Let $s^{\dag}$ be a positive integer and $ Y_{s^{\dag}} \sim {\sf N}_{n\times s^{\dag}}(0,1)$. Consider
$$
 U_{s^{\dag}}=QY_{s^{\dag}} = \frac{1}{2\pi \sqrt{-1}} \oint_\Gamma  (zB-A)^{-1}B dz Y_{s^{\dag}} = S_{(:,1:s)} (S^{-1})_{(1:s,:)} Y_{s^{\dag}}.
$$
Hence $ U_{s^{\dag}}$ is the projection of $ Y_{s^{\dag}}$ onto ${\rm span}\{S_{(:,1:s)}\}$, which implies $\rank( U_{s^{\dag}} ) \leq s$. With this in mind, if $\rank( U_{s^{\dag}} ) = s^{\dag}$, then we have $s^{\dag} \leq s$. Otherwise, if
$\rank( U_{s^{\dag}} ) < s^{\dag}$, we can conclude that $s= \rank( U_{s^{\dag}} )$ with the help of {Lemma} \ref{lem:8-8-1}, and thereby $s < s^{\dag}$.
Based on these arguments, the following algorithm was proposed in \cite{YCY14} for finding a good upper bound for $s$. Meanwhile, a projection matrix onto the eigenspace ${\rm span}\{S_{(:,1:s)}\}$ is produced, which will play an important role in the resulting method given in next section.\\

\begin{algorithm}\label{alg:2}
Input  an increasing factor $\alpha > 1$ and the size $p$ of sample vectors. The function ``\textsc{Search}" outputs $ s_1$, an upper bound of the number of eigenvalues
$s$ inside $\Gamma$, and a projection matrix $ U_1  \in {\mathbb C}^{n \times  s_1}$
onto ${\rm span}\{S_{(:,1:s)}\}$.
\end{algorithm}
\vspace{.2cm}
%\vspace{-5pt}
\begin{tabbing}
x\=xxx\= xxx\=xxx\=xxx\=xxx\=xxx\kill
\>Function $[ U_1, s_1] = \textsc{Search}( A,  B, \Gamma,  \alpha, p)$ \\
\>1. \> Pick $ Y_p \sim {\sf N}_{n \times p}(0,1)$
 and compute $\displaystyle{ U = \frac{1}{2\pi \sqrt{-1}} \oint_\Gamma  (zB-A)^{-1}B dz Y_p}$\\
\>\> by the $q$-point Gauss-Legendre quadrature rule.\\
\>2.\>  Set $s_0 = \lceil \frac{1}{p}{\rm trace}( Y_p^*  U)\rceil$ and $s^{\star} = \min(\max(p, s_0), n)$.\\
\>3.\> If $s^{\star} > p$\\
\>4.\>\> Pick $\hat{Y} \sim {\sf N}_{n \times (s^{\star} -p)}(0,1)$ and compute
$\displaystyle{\hat{U} = \frac{1}{2\pi \sqrt{-1}} \oint_\Gamma (zB-A)^{-1}B dz\hat{Y}}$ \\
\>\>\>by the $q$-point Gauss-Legendre quadrature rule.\\
\>5.\>\> Augment $\hat{U}$ to $U$ to form $ U = [ U, \hat{U} ] \in {\mathbb C}^{n \times s^{\star}}$.\\
\>6.\>Else\\
\>7.\>\>Set $s^{\star} = p$. \\
\>8.\> End\\
\>9.\> Compute $ U\Pi=U_1R_1$: the rank-revealing QR decomposition \cite{tony} of $ U$. \\
\>10.\>Set  $ s_1= {\rm rank}( R)$. \\
\>11.\> If $ s_1<s^{\star}$,  stop.  Otherwise, set $p = s_1$ and $s^{\star} = \lceil \alpha s_1\rceil$. Then go to Step 3.\\

\end{tabbing}
\vspace{.2cm}

In practice, by (\ref{eq:8-4}), we see that $U$ (line 5) formed in the last iteration in Algorithm \ref{alg:2} is
\begin{eqnarray}
U & = & \widetilde{Q}Y_{s^\star}= S\left[\frac{1}{2} \sum^{q}_{j=1}\omega_j(z_j-c)D(z_j)\right]S^{-1}Y_{s^\star}=SDS^{-1}Y_{s^{\star}} \nonumber \\
 & = &  [ D_{(1, 1)}S_{(:,1)},\ldots, D_{(d, d)}S_{(:, d)}, 0,\ldots, 0]S^{-1}Y_{s^{\star}},\label{eq:2-3-1}
\end{eqnarray}
%\begin{equation}\label{eq:2-3-1}
%U = S\left[\frac{1}{2} \sum^{q}_{j=1}\omega_j(z_j-c)D(z_j)\right]S^{-1}Y_{s^\star}=SDS^{-1}Y_{s^{\star}} = [ \widetilde{\psi}( \lambda_i)S_{(:,1)},\ldots, \widetilde{\psi}( \lambda_d)S_{(:, d)}, 0,\ldots, 0]S^{-1}Y_{s^{\star}}.
%\end{equation}
where $Y_{s^{\star}} \sim {\sf N}_{n \times s^{\star}}(0,1)$ with $s^{\star}> s_1$. According to Lemma \ref{lem:8-8-1}, we know that $S^{-1}Y_{s^\star}$ is full-rank, then it follows from (\ref{eq:2-3-1}) that
 \begin{equation}\label{eq:2-3-2}
 {\rm span}\{U_1\}= {\rm span}\{U\} = {\rm span}\{S_{(:,  \mathcal{I})}\},
\end{equation}
where $\mathcal{I}$ is an index set and its cardinality is $s_1$. It has been shown in  previous section that $\{\Re[D_{(i,i)}]> \frac{1}{2}\}_{i=1}^s$, thus we can conclude that $\{1, 2, \ldots, s\}\subset \mathcal{I}$.

\section{Counting the eigenvalues inside $\Gamma$}
In this section, we present the resulting method for counting the eigenvalues inside a given disk in the complex plane.

Below is the second finding of our work.

\begin{theorem}\label{Th:2-2-5}
Let $U_1$ be the projection matrix computed by Algorithm \ref{alg:2}.
Let
\begin{equation*}\label{eq:2-2-10}
U_2 = QU_1=\frac{1}{2\pi \sqrt{-1}} \oint_\Gamma  (zB-A)^{-1}B dz U_1.
\end{equation*}
Compute $U_2$ by the $q$-point Gauss-Legendre quadrature rule, i.e.,
\begin{equation}\label{eq:2-3-10}
U_2\approx \widetilde{U}_2 = \widetilde{Q}U_1=  \frac{1}{2} \sum^{q}_{j=1}\omega_j(z_j-c)(z_j  B- A)^{-1} BU_1,
\end{equation}
and define the $s_1\times s_1$ matrix
\begin{equation}\label{eq:2-2-11}
M = U_1^*\widetilde{U}_2.
\end{equation}
Then $\{D_{(i, i)}\}_{i=1}^s$ are the eigenvalues of $M$.
\end{theorem}
\begin{proof}
Since $U_1$ is computed by Algorithm \ref{alg:2}, by (\ref{eq:2-3-2}), we can find a nonsingular matrix, say, $W$ such that
\begin{equation*}\label{eq:2-2-12}
S_{(:,\mathcal{I})} = U_1 W.
\end{equation*}
Let $\bar{\mathcal{I}}$ be the index set $\{1, 2, \ldots, n\}\setminus \mathcal{I}$, then there exists a permutation matrix $P$ such that
\begin{equation*}\label{eq:2-2-13}
SP = [S_{(:, \mathcal{I})}, S_{(:, \mathcal{\bar{I}})}].
\end{equation*}
Note that $D$ is a diagonal matrix, thereby $P^*DP$ is also a diagonal matrix and can be written as
\begin{equation*}\label{eq:2-2-14}
P^*DP = \begin{bmatrix}
D_1     &  0   \\
 0     &  D_2
\end{bmatrix},
\end{equation*}
where the diagonal entries of $D_1$ consist of $D_{(i, i)}, i\in \mathcal{I}$.

By (\ref{eq:8-4}), (\ref{eq:2-3-10}) and (\ref{eq:2-2-11}), we have
\begin{eqnarray*}
M & = & U_1^*\widetilde{U}_2 =  U_1^* (SDS^{-1})U_1\\
 & = &  U_1^*(SP P^*DP P^*S^{-1})U_1\\
 & = & U_1^*[S_{(:, \mathcal{I})}D_1(S^{-1})_{(\mathcal{I},:)}+S_{(:, \bar{\mathcal{I}})}D_2(S^{-1})_{(\bar{\mathcal{I}},:)}](S_{(:, \mathcal{I})}W^{-1})\\
 & = & U_1^*S_{(:, \mathcal{I})}D_1W^{-1}\\
 & = &  U_1^*(U_1W)D_1W^{-1}\\
 & = & WD_1W^{-1}.
\end{eqnarray*}
Since $W$ is nonsingular, the matrices $M$ and $D_1$ have the same eigenvalues, which are $\{D_{(i, i)}\}_{i\in\mathcal{I}}$. Note that $\{1, 2,\ldots, s\}\subset \mathcal{I}$,  therefore $\{D_{(i, i)}\}_{i=1}^s$ are the eigenvalues of $M$. 
\end{proof}

Theorem \ref{Th:2-2-5} tells us that $\{D_{(i, i)}\}_{i = 1}^s$ can be obtained via computing the eigenvalues of $M$ (cf. (\ref{eq:2-2-11})). We have shown in Section 2 that the real parts of $\{D_{(i, i)}\}_{i = 1}^s$, which correspond to the eigenvalues inside $\Gamma$, are larger than $\frac{1}{2}$, and the real parts of the rest diagonal entries of $D$ are smaller than $\frac{1}{2}$. Motivated by these facts, we find the number $s$ via counting the eigenvalues of $M$ whose real parts are larger than $\frac{1}{2}$. Summarize this idea, below we give the complete algorithm for computing the number of eigenvalues inside $\Gamma$.
\vspace{0.2cm}

\begin{algorithm}\label{alg:6}
Input an increasing factor $\alpha >1$ and the size $p$ of sample vectors. The function ``\textsc{Count\_Eigs}" computes the number of eigenvalues of (\ref{eq:1-1}) that are located inside circle $\Gamma$.
\end{algorithm}
\vspace{.2cm}
%\vspace{-5pt}
\begin{tabbing}
x\=xxx\= xxx\=xxx\=xxx\=xxx\=xxx\kill
\>Function $s = \textsc{Count\_Eigs}( A,  B, \Gamma,  \alpha, p)$ \\
\>1.\> Call $[ U_1, s_1] = \textsc{Search}( A,  B, \Gamma,  \alpha, p)$.\\
\>2.\> Compute $\widetilde{U}_2$ in (\ref{eq:2-3-10}), and set $M = U_1^* \widetilde{U}_2$.\\
\>3.\>Compute the eigenvalues of $M$, and set $s$ to be the number of the computed  \\
\>\> eigenvalues whose real parts are larger than $\frac{1}{2}$.
\end{tabbing}
\vspace{.2cm}

As with other contour-integral based eigensolvers \cite{polizzi, ss, ST07, YCY14}, the dominant work of Algorithm \ref{alg:6} is solving generalized shifted linear systems of form
    \begin{equation}
\label{eq:2-4-6}
(z_jB-A)X_j = BY,
\end{equation}
which can be solved by any method of choice. Since the quadrature nodes $z_j$ and the columns of the right-hand sides of (\ref{eq:2-4-6}) are independently,
Algorithm \ref{alg:6} has good scalability in modern parallel architectures.

The information about the number of eigenvalues inside the target region is crucial to the implementation of the contour-integral based eigensolvers. Algorithm \ref{alg:6} can integrate with this kind of eigensolvers \cite{polizzi,  ss, ST07, YCY14} to provide them with this important information.  For example, if we use the contour-integral based eigensolver proposed in \cite{YCY14} to compute the eigenvalues inside $\Gamma$, $U_1$ and $\widetilde{U}_2$ need to be computed in the first iteration and the second, respectively. As a result, to get $s$, the extra work is constructing the matrix $M$ (cf. (\ref{eq:2-2-11})) and then computing the eigenvalues of $M$. Once the number $s$ is obtained, it can help to detect whether all eigenvalues inside $\Gamma$ are found in the subsequent iterations and determine when to stop the iteration process. For clarity, we summarize the idea of integrating Algorithm \ref{alg:6} with the eigensolver proposed in \cite{YCY14} by the following algorithm.

\begin{algorithm}\label{alg:5}
Input $ A,  B \in {\mathbb C}^{n \times n}$, an increasing factor $\alpha >1$, the size $p$ of sample vectors, a circle $\Gamma$, a convergence tolerance
$\epsilon$,  and ``max\_iter" to control the
maximum number of iterations. The function ``{\sc Eigenpairs}"
computes eigenpairs $(\tilde{\lambda}_i, \tilde{{\bf x}}_i)$ of (\ref{eq:1-1}) that satisfies
\begin{equation}\label{con-cre_2}
\tilde{\lambda}_i \ {\rm inside} \ \Gamma
\quad {\rm and} \quad \frac{\| A\tilde{{\bf x}}_i - \tilde{\lambda}_i  B\tilde{{\bf x}}_i\|_2}{\| A \tilde{{\bf x}}_i\|_2+
\| B \tilde{{\bf x}}_i\|_2} < \epsilon.
\end{equation}
The results are stored in vector $\Lambda$ and matrix $X$. \end{algorithm}
\vspace{.2cm}
%\vspace{-5pt}
\begin{tabbing}
x\=xxx\= xxx\=xxx\=xxx\=xxx\=xxx\kill
\> Function $[ \Lambda,  X] = \textsc{Eigenpairs}( A,  B, \Gamma, \epsilon, max\_iter)$\\
\>1.\> Call $[ U_1, s_1] = \textsc{Search}( A,  B, \Gamma,  \alpha, p)$.\\
\>2.\> Compute $\widetilde{U}_2$ in (\ref{eq:2-3-10}), and let $M = U_1^* \widetilde{U}_2$. Set $s$ to be the number of the \\
\>\>eigenvalues of $M$ whose real parts are larger than $\frac{1}{2}$.\\
\>3.\>For $k = 2,\cdots, max\_iter$\\
\>4.\>\> Compute QR decompositions: $\widetilde{U}_k = U_1R_1$ and $B\widetilde{U}_k = U_2R_2.$\\
\>5.\>\>Form $ \widetilde{A} =  U_2^*  A  U_1$ and $\widetilde{B} =  U_2^*  B  U_1$.\\
\>6.\>\> Solve the projected eigenproblem $\widetilde{A} {\bf y} = \tilde{\lambda} \widetilde{B} {\bf y}$ of size $ s_1$ to obtain eigenpairs\\
\>\>\>$\{(\tilde{\lambda}_i, {\bf y}_i)\}_{i=1}^{s_1}$. Set  $ \tilde{{\bf x}}_i =  U_1{\bf y}_i, i = 1, 2, \ldots,  s_1$. \\
\>7.\>\> Set $ \Lambda=\left[ \  \right]$ and $X=\left[ \  \right]$.\\
\>8.\>\> For $i = 1: s_1$\\
\>9.\>\>\>If $(\tilde{\lambda}_i,\tilde{ {\bf x}}_i)$ satisfies (\ref{con-cre_2}), then $ \Lambda = [\Lambda, \tilde{\lambda}_i]$ and $ X =[X, \tilde{{\bf x}}_i]$.\\
\>10.\>\>End\\
\>11.\>\>If there are $s$ eigenpairs satisfying (\ref{con-cre_2}), stop. Otherwise, compute\\
\>\>\> $U_{k+1} =\frac{1}{2\pi \sqrt{-1}} \oint_\Gamma  (zB-A)^{-1}B dz U_1$ by the $q$-point Gauss-Legendre \\
\>\>\>quadrature rule: $
U_{k+1}\approx \widetilde{U}_{k+1} =  \frac{1}{2} \sum^{q}_{j=1}\omega_j(z_j-c)(z_j  B- A)^{-1} BU_1$.\\
\>12.\>End \\
\end{tabbing}

Below we give some remarks on above algorithm.
\begin{itemize}
  \item[1.] The first two steps can be viewed as using Algorithm \ref{alg:6} to determine the number of eigenvalues inside $\Gamma$. In addition, a projection matrix is also produced, i.e., $\widetilde{U}_2$.
   \item[2.] Since $\widetilde{U}_2$ is a projection matrix onto ${\rm span}\{S_{(:, 1:s)}\}$, it can be used to construct a projected eigenproblem to compute approximate eigenpairs. Therefore, in Step 3, the for-loop starts from $k=2$.
  \item[3.] In Step 11,  the number $s$ computed in Step 2 helps to detect whether all $s$ approximate eigenpairs satisfying the prescribed accuracy. If it is, we stop the iteration process.
\end{itemize}

\section{Numerical Experiments}\label{sec:experiments}
In this section, we give some numerical experiments to illustrate the viability
of our new method.  All computations are carried out in \textsc{Matlab}
version R2012b on a MacBook with an Intel Core i5 2.5 GHz processor and 8 GB RAM.

In the experiments, as for computing the generalized shifted linear systems of the form (\ref{eq:2-4-6}), we first use the \textsc{Matlab} function \texttt{lu} to compute the LU decomposition of $z_j B-A,  j=1, 2, \ldots, q$, and then perform the triangular substitutions to get the corresponding solutions.

\textsf{Experiment\ 5.1:}
Our method (Algorithm \ref{alg:6}) is motivated by two findings: (i) for matrix $D$ (cf. (\ref{eq:8-1-1})), $\Re[D_{(i, i)}]  > \frac{1}{2}$ if $D_{(i, i)}$ correspond to the eigenvalues enclosed by $\Gamma$,  and $\Re[D_{(i, i)}]$ are less than $\frac{1}{2}$ if $D_{(i, i)}$ correspond to the eigenvalues outside $\Gamma$;
(ii) $D_{(i, i)}$ are the eigenvalues of $M$ (cf. (\ref{eq:2-2-11})) if they correspond to the eigenvalues within $\Gamma$.  This experiment is devoted to illustrating these two findings.

Let $\Lambda = \texttt{diag}([0.1:0.1:0.8])$, $S = \texttt{randn}(8)$, the matrices $A$ and $B$ are given by
\begin{equation*}\label{eq:5-1-1}
A = S\Lambda S^{-1},\qquad B = \texttt{eye}(8).
\end{equation*}
Here \texttt{diag}, \texttt{randn}, and \texttt{eye} are \textsc{Matlab} commands.
Obviously, for the problem under consideration, the eigenvalues are $0.1, 0.2, \ldots, 0.8$. Let $\Gamma$ be a circle with center at origin and radius $\rho = 0.401$. Suppose that we are interested in the number of eigenvalues inside $\Gamma$. Obviously, there are 4 eigenvalues are located inside $\Gamma$. Note that the eigenvalue $0.4$ is located inside $\Gamma$ and close to boundary of the disk surrounded by $\Gamma$.

In this experiment, we select the number of quadrature points $q = 32$. According to  (\ref{eq:8-1-1}), now $D$ is given by
\begin{equation}\label{eq:5-1-2}
D = \sum_{i=1}^{32}\omega_i z_{i}(z_i I_8 -\Lambda)^{-1},
\end{equation}
where $\omega_i$ are the weights associated with quadrature nodes $z_i$.
We take the size of sample vectors $p = 6$, thus the starting basis in function \textsc{Search} is $Y_6 =  \texttt{randn}(8, 6)$. Since the size of test problem is small and the number of columns of $Y_6$ is already larger than the number of eigenvalues inside $\Gamma$, we just run one iteration when preforming function \textsc{Search} to get projection matrix $U_1$. As a result, the matrix $M$ defined in (\ref{eq:2-2-11}) is of size $6\times 6$.

\begin{table}
\centering
\caption{The real parts of the diagonal entries of $D$ and the ones of the eigenvalues of $M$.}
\footnotesize{
%{\newcommand{\q}[1]{\mc{1}{|l|}{\small\tt #1}}
\noindent
\begin{tabular}{c|cc}
$i$& $\Re[(D_{(i,i)})]$ &  $\Re[\texttt{eig}(M)]$\\
\hline
$1$ & $\underline{1.0000000000039}49$& $\underline{1.0000000000039}65$  \\
$2$ & $\underline{1.0000000000000}00$ & $\underline{1.0000000000000}12$ \\
$3$ & $\underline{0.999999999999999}$& $\underline{0.999999999999999}$ \\
$4$ & $\underline{0.8015817876596}01$& $\underline{0.8015817876596}10$ \\
$5$ & $\underline{0.0000000025256}84$& $\underline{0.0000000025256}20$ \\
$6$ & $\underline{0.0000000000043}79$& $\underline{0.0000000000043}80$  \\
$7$ & $0.000000000000001$& \\
$8$ & $-0.000000000000051$&   \\
\end{tabular}}
\label{Tab:5-1-1}
\end{table}

Since, in practical situations, we are only interested in the real parts of the diagonal entries of $D$, in the second column in \textsc{Table} \ref{Tab:5-1-1} we list $\{\Re[D_{(i, i)}]\}_{i=1}^8$ that are computed by (\ref{eq:5-1-2}). It can be seen that the $\{\Re[D_{(i, i)}]\}_{i=1}^4$, corresponding to the eigenvalues inside $\Gamma$, are larger than $0.5$.  More precisely, $\Re[D_{(i, i)}], i = 1, 2, 3$, corresponding to eigenvalues 0.1, 0.2, 0.3, respectively, approximate the theoretical value 1 sufficiently. But $\Re[D_{(4, 4)}]$ is about $0.8$, this is because it corresponds to the eigenvalue $0.4$, which is close to the boundary of the target disk. On the other hand, the real parts of $\{D_{(i, i)}\}_{i =5}^8$, corresponding to the eigenvalues outside $\Gamma$, are less than 0.5 and very close to zero, as expected of our first finding.

The third column in \textsc{Table} \ref{Tab:5-1-1} displays the real parts of the eigenvalues of $M$ in descent order. The same digits of $\Re[D_{(i, i)}]$ and the $i$th largest $\Re[\texttt{eig}(M)]$, $i = 1, \ldots, 6$ are underlined. We can see that $\Re[D_{(i, i)}], i = 1, \ldots, 6 $, agree at least fourteen digits to their counterparts in the third column. Therefore the real parts of the eigenvalues of $M$ are almost equivalent to $\Re[D_{(i,i)}], i = 1,\ldots, 6$, which justifies the second finding of our work.

\textsf{Experiment\ 5.2:}
This experiment is devoted to testing the viability of our new method.
The test matrices are available from the Matrix Market collection \cite{DGL89}. They are the real-world problems from scientific and
engineering applications. The descriptions of the matrices are presented in
\textsc{Table} \ref{Tab:5-1}, where nnz denotes the number of non-zero entries
and their condition numbers are computed by \textsc{Matlab} function \texttt{condest}.
The test matrices of different problems vary in size, spectrum and property.

\begin{table}
\caption{Test problems from Matrix Market that are used in Experiment 5.2}
\footnotesize{
%{\newcommand{\q}[1]{\mc{1}{}{\small\tt #1}}
\noindent
\begin{tabular}{c|llllc}
No.&Matrix & Size&nnz & Property & \texttt{cond}
\\ \hline
1& $ A$: BFW398A & $398$& $3678$&  unsymmetric &$7.58\times 10^{3}$\\
& $ B$: BFW398B  & $398$& $2910$& symmetric indefinite &$3.64\times 10^{1}$\\ \hline
2& $ A$: BFW782A& $782$ & $7514$& unsymmetric &$4.63\times 10^{3}$\\
& $ B$: BFW782B & $782$& $5982$&  symmetric indefinite &$3.05\times 10^{1}$\\ \hline
3 & $ A$: BCSSTK08 & $1074$ &12960 &symmetric positive definite & $4.77\times 10^{7}$\\
 & $ B$: BCSSTM08 & $1074$ & 1074& symmetric positive definite & $8.27\times 10^{6}$\\ \hline
 4&$ A$: BCSSTK27& $1224$& 28675& symmetric positive definite & $7.71\times 10^{4}$\\
&$ B$: BCSSTM27 & $1224$& 28675& symmetric indefinite & $1.14\times 10^{10}$\\ \hline
5&$ A$: PLAT1919 & $1919$& 17159&  symmetric indefinite & $1.40\times 10^{16}$\\
&$ B$: PLSK1919 & $1919$& 4831& skew symmetric & $1.07\times 10^{18}$\\ \hline
6&$ A$: BCSSTK13& $2003$& 42943& symmetric positive definite &$4.57\times 10^{10}$\\
&$ B$: BCSSTM13 & $2003$& 11973& symmetric positive semi-definite &Inf\\ \hline
7&$ A$: MHD4800A & $4800$& 102252& unsymmetric & $2.54\times 10^{57}$\\
& $ B$: MHD4800B & $4800$&  27520& symmetric indefinite &$1.03\times 10^{14}$\\ \hline
8&$ A$: BCSSTK25  & $15439$& 252241& symmetric indefinite & $1.28\times 10^{13}$\\
& $ B$: BCSSTM25  & $15439$&  15439& symmetric positive definite &$6.06\times 10^{9}$
\end{tabular}}
\label{Tab:5-1}
\end{table}

In this experiment, we take the parameter $q$ to be $16$. The parameters $c$ and $\rho$ are the center and the radius of circle $\Gamma$, respectively. Note that the test problems 3, 4, 6, and 8 are Hermitian problems, which means their (finite) eigenvalues are real-valued. Due to this, we choose the circles with centers lying on the real line for these test problems.

\textsc{Table} \ref{Tab:5-3} presents the numerical comparisons. $s$ is the actual number of  eigenvalues inside $\Gamma$. We first use the \textsc{Matlab} built-in function \texttt{eig} to compute all eigenvalues of the test problems, and then determine the values of $s$ according to the coordinates of computed eigenvalues. We also present the estimation $s_0$  computed by the trace formula (\ref{eq:5-17-1}) and the upper bound $s_1$ computed by Algorithm \ref{alg:2}.  Cont\_Eigs is the result computed by our new method.

The results for all eight test problems are reported in \textsc{Table} \ref{Tab:5-3}. From these data, we see that the estimation $s_0$ computed by the trace formula (\ref{eq:5-17-1}) always provide good estimation to $s$ for all test problems except for test problem 7, due to its ill-condition. $s_1$ always gives a good upper bound for $s$. It is remarkable that the result computed by our new method is the same with the exact number $s$ for each test problem, even though the test problem 7. Therefore, our new method is numerically efficient and reliable.

\begin{table}
\centering
\caption{Numerical comparison: $s$ is the exact number of eigenvalues inside $\Gamma$, $s_0$ is the estimate of $s$ by using the trace formula and $s_1$ is the upper bound computed by {Algorithm} \ref{alg:2}, and Cont\_Eigs is the result computed by our new method.}
\footnotesize{
%{\newcommand{\q}[1]{\mc{1}{|l|}{\small\tt #1}}
\noindent
\begin{tabular}{c|ccccc}
No.&$(c, \rho) $&$s$ & $s_0$&$s_1$& Cont\_Eigs  \\ \hline
1& $((-6.0\times 10^{5})+\sqrt{-1} (2.0\times 10^{5}), 3.0\times 10^5)$ &120 & 136& 177 &120  \\
2&  $((-5.0\times 10^{5})+\sqrt{-1} (1.0\times 10^{5}), 2.0\times 10^5)$  & 165& 161 &228 &165\\
3&$(5.0\times 10^{5}, 3.0\times 10^{5})$  &178  &199 &190 & 178 \\
4& $(5.0\times 10^{3}, 3.0\times 10^{3})$& 160 & 134 &192 &160  \\
5& $(\sqrt{-1} (5.0\times 10^{-1}), 1.5\times 10^{-1})$&   301& 297 & 397 & 301\\
6 & $(6.0\times 10^{6}, 3.5\times 10^{6})$   & 232& 229 &262 &232 \\
7& $((-5.0\times 10^{2})+\sqrt{-1} (2.0\times 10^{2}), 4.0\times 10^2)$ & 212 & 545 &293 &212 \\
8& $(5.0\times 10^{5}, 2.0\times 10^{5})$ & 1663 & 1628 & 1749 & 1663
\end{tabular}}
\label{Tab:5-3}
\end{table}

\section{Conclusion} In this work, we develop an approach for counting the eigenvalues of (\ref{eq:1-1}) inside a given disk in the complex plane. The new method is a contour-integral based method and motivated by two findings. The computational advantage of the new method is that it is easily parallelizable. Its another promising feature is that it can integrate with the recently proposed contour-integral based eigensolvers to provide them the information of the number of eigenvalues inside the target region. How to adapt the resulting method to the nonlinear problems will be our future work.

\end{document}